\documentclass[12pt,twoside]{amsart}
\usepackage{amssymb}
\usepackage{graphicx}
\usepackage[margin=1in]{geometry}

\newcommand{\Z}{\mathbb{Z}}
\newcommand{\R}{\mathbb{R}}
\newcommand{\N}{\mathbb{N}}
\newcommand{\Q}{\mathbb{Q}}

\newcommand {\PP}{\mathbb{P}}
\newcommand{\cc}{\mathbb{C}}
\newcommand{\nn}{\mathbb{N}}
\newcommand{\zz}{\mathbb{Z}}
\newcommand{\beas}{\begin{eqnarray*}}
\newcommand{\eeas}{\end{eqnarray*}}
\newcommand{\bea}{\begin{eqnarray}}
\newcommand{\eea}{\end{eqnarray}}
\newcommand{\beq}{\begin{equation}}
\newcommand{\eeq}{\end{equation}}
\newcommand{\ben}{\begin{enumerate}}
\newcommand{\een}{\end{enumerate}}

\newtheorem{theorem}{Theorem}[section]
\newtheorem{lemma}[theorem]{Lemma}
\newtheorem{proposition}[theorem]{Proposition}

\theoremstyle{definition}
\newtheorem{remark}[theorem]{Remark}
\newtheorem{example}[theorem]{Example}

\usepackage{latexsym}
\usepackage{color,hyperref}
\definecolor{darkblue}{rgb}{0,0,0.6}
\hypersetup{colorlinks,breaklinks,linkcolor=darkblue,urlcolor=darkblue,anchorcolor=darkblue,citecolor=darkblue}

\title[Some asymptotic results on $q$-binomial coefficients]{Some asymptotic results on $q$-binomial coefficients}

\author[Richard P. Stanley]{Richard P. Stanley${}^{\ast }$} \address{Department of Mathematics\\ MIT\\ Cambridge, MA 02139-4307}
\email{rstan@math.mit.edu}

\author[Fabrizio Zanello]{Fabrizio Zanello${}^{\ast \ast}$} \address{Department of Mathematical  Sciences\\ Michigan
  Tech\\ Houghton, MI  49931-1295} 
\email{zanello@mtu.edu}

\thanks{2010 {\em Mathematics Subject Classification.} Primary: 05A16;
  Secondary: 05A17.\\\indent 
{\em Key words and phrases.} $q$-binomial coefficient; asymptotic
enumeration; integer partition; Eulerian number; Euler-Frobenius
number; Kostka number.\\\indent  
${}^*$This author's contribution is based upon work supported by the National
    Science Foundation under Grant No.~DMS-1068625.\\\indent
    ${}^*{}^*$This author is partially supported by a
Simons Foundation grant (\#274577).} 
 
\begin{document}

\begin{abstract} 
We look at the asymptotic behavior of the coefficients of the
$q$-binomial coefficients (or Gaussian polynomials)
$\binom{a+k}{k}_q$, when $k$ is fixed. We give a number of results in
this direction, some of which involve Eulerian polynomials and their
generalizations.
\end{abstract}

\maketitle

\section{Introduction}


The purpose of this note is to investigate the asymptotic behavior of
the coefficients of the $q$-binomial coefficient (or Gaussian
polynomial) $\binom{a+k}{k}_q$. While much of the previous work in
this area has focused on the case where both $a$ and $k$ get
arbitrarily large (see e.g. \cite{takacs}), in this paper we will be
concerned with asymptotic estimates for the coefficients of
$\binom{a+k}{k}_q$ when $k$ is fixed.

Besides the intrinsic relevance of studying the combinatorial,
analytic or algebraic properties of $q$-binomial coefficients, our
work is also motivated by a series of recent papers that have revived
the interest in analyzing the behavior of the coefficients of
$\binom{a+k}{k}_q$, as well as their applications to other
mathematical areas. See for instance \cite{PP}, where I. Pak and
G. Panova have first shown algebraically the strict unimodality of
$\binom{a+k}{k}_q$, as well as the subsequent combinatorial proofs of
the Pak-Panova result by the second author of this paper \cite{Za} and
by V. Dhand \cite{Dh}. See also another interesting recent work by Pak
and Panova \cite{PP2} (as well as their extensive bibliography), where
the coefficients of $\binom{a+k}{k}_q$ have been investigated in
relation to questions of representation theory concerning the growth
of Kronecker coefficients. Further, one of the results of this note,
Theorem~\ref{thm:dkx}, has also been motivated by, and finds a first
useful application in the study of the unimodality of partitions with
distinct parts that are contained inside certain Ferrers diagrams (see
our own paper \cite{SZ}).

For $m=\lfloor ak/2 \rfloor$ (the middle exponent of
$\binom{a+k}{k}_q$ when $k$ or $a$ are even, and the smaller of the
two middle exponents otherwise), define $g_{k,c}(a)$ to be the
coefficient of degree $m-c$ of $\binom{a+k}{k}_q$, and let
$f_{k,c}(a)= g_{k,c}(a)-g_{k,c+1}(a)$. Our first main result is a
description of the generating functions (in two variables, referring
to $a$ and $c$) of $g_{k,c}(a)$ and $f_{k,c}(a)$. In particular, it
follows from our result that both $g_{k,c}(a)$ and $f_{k,c}(a)$ are
quasipolynomials in $a$, for any given $k$ and $c$.

Our next result, Theorem~\ref{thm:calk}, is an asymptotic estimate of
the coefficient of degree $\lfloor \alpha a\rfloor-c$ of
$\binom{a+k}{k}_q$, when $a\to\infty$, for any given integer $c$,
positive integer $k$, and nonnegative real number $\alpha$. Quite
surprisingly, this result connects in a nice fashion to Eulerian
numbers and, more generally, to Euler-Frobenius numbers, as we will
discuss extensively after the proof of the theorem.

Finally, our last main result, Theorem~\ref{thm:dk}, presents an
asymptotic estimate of the difference between consecutive coefficients
of $\binom{a+k}{k}_q$, again for $k$ fixed.

We will wrap up this note with a brief remark, in order to highlight
an interesting connection of our last result with Kostka numbers and
to present some suggestions for further research.

\section{Some asymptotic properties of the coefficients of
  $\binom{a+k}{k}_q$}  
In this section, we study the asymptotic behavior of the coefficients of
$\binom{a+k}{k}_q$ for fixed $k$. Given $k\geq 1$, $c\geq 0$, and
$a\geq 0$, set $m=\lfloor ak/2 \rfloor$. Define
  \bea g_{k,c}(a) & = & [q^{m-c}]\binom{a+k}{k}_q,\nonumber \\
     f_{k,c}(a) & = & g_{k,c}(a)-g_{k,c+1}(a), \label{eq:fkc} \eea
where $[q^n]F(q)$ denotes the coefficient of $q^n$ in the polynomial
(or power series) $F(q)$. 
 
\begin{lemma}\label{lemma:dis}
Let $F(q)\in\cc[[q]]$, and $c,j,i\in \Z$ with $j>i\geq 0$. We have:
\ben\item[(a)]
  $$ \sum_{a\geq 0} [q^{aj-c}]q^{ai}F(q)x^a =
  \frac{1}{j-i}\left.\sum_{\zeta^{j-i}=1} (\zeta x)^c F(\zeta
  x)\right|_{x\to x^{1/(j-i)}} $$
\item[(b)] 
  $$ \sum_{a\geq 0}\sum_{c\geq 0}[q^{aj-c}]q^{ai}F(q)x^a t^c =
   \frac{1}{j-i}\left.\sum_{\zeta^{j-i}=1}\frac{F(\zeta x)}{1-\zeta xt}
   \right|_{x\to x^{1/(j-i)}}. $$
\een
\end{lemma}

\begin{proof}

For any $G(q)=\sum a_iq^i\in\cc[[q]]$ and $h\geq 1$, write
  $$ D^h G(q) = \sum a_{hi}x^{hi}, $$
the $h$th \emph{dissection} of $G(q)$.
It is an elementary and standard result (see e.g. \cite[Exercise
1.60]{St0}) that  
   $$ D^h G(q) = \frac 1h \sum_{\zeta^h=1}
       G(\zeta x). $$
(The sum is over all $h$ complex numbers $\zeta$ satisfying
$\zeta^h=1$.) Hence (a) follows.
 
 Part (b) is the generating function (in $t$) with respect to $c$ of
 the formula of part (a). We have: 
  \beas \sum_{a\geq 0}\sum_{c\geq
    0}[q^{aj-c}]q^{ai}F(q)x^at^c
  & = & \sum_{a\geq 0} \sum_{c\geq 0} [q^{a(j-i)}]q^cF(q)x^at^c\\
   & = & \sum_{a \geq 0} [q^{a(j-i)}]\frac{F(q)}{1-qt}x^a\\
   & = & \frac{1}{j-i}\left.\sum_{\zeta^{j-i}=1}\frac{F(\zeta x)}{1-\zeta
    xt}\right|_{x\to x^{1/{(j-i)}}}, \eeas
and the proof follows.
\end{proof}

From Lemma~\ref{lemma:dis}, it is easy to describe the form of the
generating functions for $g_{k,c}(a)$ and $f_{k,c}(a)$, when $k$ and
$c$ are fixed. For this purpose, define a \emph{quasipolynomial} to be
a function $h\colon \nn\to\cc$ (where $\nn=\{0,1,2,\dots\}$) of the
form
   $$ h(n)=c_d(n)n^d +c_{d-1}(n)n^{d-1}+\cdots+ c_0(n), $$ 
where each $c_i(n)$ is a periodic function of $n$. If $c_d(n)\neq 0$
then we call $d$ the \emph{degree} of 
$h$. For more information on quasipolynomials, see for instance
\cite[{\S}4.4]{St0}.

Write
  \beas F_k(x,t) & = & \sum_{a\geq 0}\sum_{c\geq 0} f_{k,c}(a)x^at^c\\
 G_k(x,t) & = & \sum_{a\geq 0}\sum_{c\geq 0} g_{k,c}(a)x^at^c. \eeas

\begin{theorem} \label{thm:dkx}
Fix $k\geq 1$ and set $j=\lfloor k/2\rfloor$. If we denote both $F_k$
and $G_k$ by $H_k$, then 
  $$ H_k(x,t) = \left\{ \begin{array}{ll} \displaystyle \frac{N_k(x,t)}
     {D_k(x)(1-tx)(1-t^2x)(1-t^3x)\cdots(1-t^jx)}, & k\ \mathrm{even}\\[1em]
      \displaystyle \frac{N_k(x,t)}
     {D_k(x)(1-tx^2)(1-t^3x^2)(1-t^5x^2)\cdots(1-t^kx^2)}, 
      & k\ \mathrm{odd}. \end{array} \right. $$
where $N_k(x,t)\in\zz[x,t]$ and $D_k(x)$ is a product of cyclotomic
polynomials. In particular, for fixed $k$ and $c$ we have that 
$g_{k,c}(a)$ and $f_{k,c}(a)$ are quasipolynomials.  
\end{theorem}

\begin{proof}
\emph{Case 1:} $k=2j$. We have $m=\lfloor ak/2 \rfloor=aj$ . Write
  \beq (1-q^{a+1})(1-q^{a+2})\cdots (1-q^{a+k})=\sum_{i=0}^k
    (-1)^i P_i(q)q^{ai}, \label{eq:piq} \eeq 
where $P_i(q)$ is a polynomial in $q$ independent of
$a$. Specifically, we have
  \beq P_i(q)=\sum_{\substack{S\subseteq [k]\\ \#S=i}}
    q^{\sum_{s\in S}s}. \label{eq:piq2} \eeq
Writing $[k]!=(1-q)(1-q^2)\cdots(1-q^k)$, we get
  \beas G_k(x,t) & = & \sum_{a\geq 0}\sum_{c\geq
    0}[q^{m-c}]\binom{a+k}{k}_q x^a t^c\\ & = &
   \sum_{a\geq 0}\sum_{c\geq 0}[q^{aj-c}]\frac{1}{[k]!}
     \sum_{i=0}^k (-1)^iP_i(q)q^{ai}x^at^c\\ & = &
   \sum_{i=0}^k\sum_{a\geq 0}\sum_{c\geq 0}[q^{aj-c}]
      \frac{1}{[k]!} (-1)^iP_i(q)q^{ai}x^at^c. \eeas      

The proof now follows from Lemma~\ref{lemma:dis}(a). 
Note in particular that the expression $F(\zeta x)$ in
Lemma~\ref{lemma:dis} will produce cyclotomic polynomials in the
denominator of $G_k(x,t)$, while the denominator $1-\zeta xt$ in part
(b) will lead to the factor $1-t^{j-i}x$ in the denominator of
$G_k(x,t)$. The proof for $F_k(x,t)$ is completely analogous.

\emph{Case 2:} $k=2j+1$. The proof is analogous to Case~1. Now we have
to look at $a=2b$ and $a=2b+1$ separately. When $a=2b$ we get that the
part of $G_k(x,t)$ with even exponent of $x$ is 
  $G_k(x,t) = \sum_{a\geq 0}\sum_{c\geq
  0}[q^{bk-c}]\binom{2b+k}{k}x^{2b}t^c$.
When we apply Lemma~\ref{lemma:dis}, the denominator term 
becomes $1-\zeta x^2t$, where $\zeta^{j-i}=1$ and $j-i$ is odd. This
produces a factor $1-t^{j-i}x^2$ (where $j-i$ is odd) in the
denominator of $G_k(x,t)$. Exactly the same reasoning applies to
$a=2b+1$, so the proof follows.
\end{proof}

\begin{example}
Write $\Phi_m(x)$ for the $m$th cyclotomic polynomial normalized to
have constant term 1. Hence $\Phi_1(x)=1-x$, $\Phi_2(x)=1+x$,
$\Phi_3(x)=1+x+x^2$, etc.. One can compute the following:
 \bea
 F_3(x,t) & = & \frac{1+tx+tx^3+t^3x^4}
   {(1-x)(1+x)(1+x^2)(1-tx^2)(1-t^3x^2)} \nonumber\\
 G_3(x,t ) & = & \frac{N_3(x,t)}{(1-x)^2(1-x^4)(1-tx^2)(1-t^3x^2)}
    \nonumber \\
 F_4(x,t) & = & \frac{1-tx+t^2x^2}{(1-x^2)(1-x^3)(1-tx)(1-t^2x)}
   \label{eq:f4xt} \\
 G_4(x,t) & = & \frac{1-x+(1+t)x^2-(t+t^2)x^3}
   {(1-x)^2(1-x^2)(1-x^3)(1-tx)(1-t^2x)} \nonumber \\
 F_5(x,0) & = & \frac{1-x^5-x^6+x^7+x^{12}} 
   {\Phi_1^3\Phi_2^3\Phi_3\Phi_4^2\Phi_6\Phi_8}   \nonumber \\
 G_5(x,0) & = &
 \frac{B_5(x)}{(1-x)^2(1-x^4)(1-x^6)(1-x^8)} \nonumber
  \\ 
 F_6(x,t) & = & \frac{M_6(x,t)}{\Phi_1^4\,\Phi_2^2\,\Phi_3\,\Phi_4\,
   \Phi_5\,(1-tx)(1-t^2x)(1-t^3x)} \nonumber \\
 G_6(x,t) & = &
 \frac{N_6(x,t)}{\Phi_1^6\,\Phi_2^3\,\Phi_3\,\Phi_4\,\Phi_5\,
     (1-tx)(1-t^2x)(1-t^3x)}, \nonumber 
 \eea
where
 $$ N_3(x,t) = 1-(1-t)x+(1-t+t^2)x^2 +(t-t^2)x^3 -
    (t-t^2)x^4-(t^2+t^3)x^5 $$

 $$ B_5(x) =
    1-x+2x^2+x^3+2x^4+3x^5+x^6+5x^7+x^8+3x^9+2x^{10}+x^{11}+ 
       2x^{12} $$  
  $$ \qquad -x^{13}+x^{14}
    +3x^{10}+x^{12}-x^{13}+2x^{14}-x^{15}+x^{17}-2x^{18}+x^{19} $$  

 $$ M_6(x,t) = 1+(1-t-t^2)x-(t-t^3-t^4)x^2-(1-t-t^2-t^3-t^4+t^5)x^3 $$
 $$
-(1-2t-t^2+t^5)x^4-(1-2t-t^2+t^3+t^4)x^5+(t+t^2-t^3-2t^4+t^5)x^6 $$
  $$+(1-t^3-2t^4+t^5)x^7+(1-t-t^2-t^3-t^4+t^5)x^8-(t+t^2-t^4)x^9 $$
   $$ +(t^3+t^4-t^5)x^{10}-t^5x^{11} $$

 $$
N_6(x,t)=1+(1+2t+t^2)x^2+(3+2t-t^2-2t^3-t^4)x^3+(4-2t^2-3t^3-t^4+t^5)x^4\\ $$
  $$ +(4-3t^2-4t^3-t^4+2t^5)x^5+(4-t-4t^2-4t^3-t^4+3t^5)x^6 $$
  $$+(3-t-5t^2-4t^3+3t^5)x^7+(1-t-4t^2-3t^3+t^4+4t^5)x^8 $$
  $$ -(2t^2+t^3-t^4-3t^5)x^9+(1-t^2-t^3+3t^5)x^{10} $$
  $$ -(t+t^2+t^3-t^4-2t^5)x^{11}+(t^3+t^4+t^5)x^{12}. $$
  
The denominator of $F_8(x,t)$ is given by
  $$ \Phi_1^6\,\Phi_2^3\,\Phi_3^2\,\Phi_4\,\Phi_5
    \,\Phi_7\,(1-tx)(1-t^2x)(1-t^3x)(1-t^4x), $$
and that of $G_8(x,t)$ by
  $$ \Phi_1^8\,\Phi_2^3\,\Phi_3^2\,\Phi_4\,\Phi_5\,\Phi_7\,
    (1-tx)(1-t^2x)(1-t^3x)(1-t^4x). $$
    
Let us also note that 
  \beas F_8(x,0) & = & \sum_{a\geq
    0}([q^{4a}]-[q^{4a-1}])\binom{a+8}{8}_q x^a\\ & = &
   \frac{1+x-x^3-x^4+x^6+x^7+x^8+x^9+x^{10}-x^{12}-x^{13}+
   x^{15}+x^{16}}{(1+x)(1-x^2)(1-x^3)^2(1-x^4)(1-x^5) 
       (1-x^7)}\\[.2em] & \hspace{-3.5em} = & \hspace{-2em} 
   1+x^2+x^3+2x^4+2x^5+4x^6+4x^7+7x^8+8x^9+12x^{10}+\cdots. 
    \eeas
    
This generating function appears in a paper \cite[p.~847]{igusa} of
Igusa, stated in terms of the representation theory of
SL$(n,\cc)$. Igusa also computes $F_2(x,0)$, $F_4(x,0)$, and
$F_6(x,0)$.
\end{example}

From the techniques for computing $F_k(x,t)$ and $G_k(x,t)$, we can
determine asymptotic properties of some of the coefficients of
$\binom{a+k}{k}_q$, for $k$ fixed. The coefficients of
$\binom{a+k}{k}_q$ have been considered for $a,k\to \infty$ by
Tak\'acs \cite{takacs} and others, but the computation for $k$ fixed
seems to be new. As usual, we define $f(x)=O(g(x))$ for $x\to x_0\le \infty$, if $\vert f(x)\vert \le C\cdot \vert g(x)\vert$ for some constant $C>0$, when $x$ approaches $x_0$.

\begin{theorem} \label{thm:calk}
Fix $\alpha\geq 0$ ($\alpha\in\R$), $c\in \Z$, and $k$ a positive integer. Then
  $$ [q^{\lfloor \alpha a\rfloor-c}]\binom{a+k}{k}_q = 
    \frac{1}{(k-1)!\,k!}C(\alpha, k)a^{k-1}+O(a^{k-2}) $$
for $a\to \infty$, where
   $$ C(\alpha,k) = \sum_{i=0}^{\lfloor\alpha\rfloor}
      (-1)^i\binom ki(\alpha-i)^{k-1}. $$
\end{theorem}

\begin{proof}
First assume that $\alpha$ is rational, say $\alpha=u/v$. Fix $0\leq
r<v$ and consider only those $a$ of the form
$a=vb+r$. Set $d=\lfloor ur/v\rfloor$.  Thus
  \beq \left[ q^{\lfloor ua/v\rfloor-c}\right] \binom{a+k}{k}_q =
    [q^{ub+d-c}]\frac{(1-q^{vb+r+k})(1-q^{vb+r+k-1})\cdots
     (1-q^{vb+r+1})}{(1-q^k)(1-q^{k-1})\cdots (1-q)}.
   \label{eq:uav} \eeq
   
Write
  \beas G_{\alpha,k,r}(x) & = &\sum_{\substack{a\geq 0\\ a\equiv
      r\,(\mathrm{mod}\,v)}} \left[ q^{\lfloor 
    ua/v\rfloor-c}\right] \binom{a+k}{k}_q x^a\\
    & = & \sum_{b\geq 0}[q^{ub +d-c}]\binom{vb+r+k}{k}_q
      x^{vb+r}. \eeas   
      
We now apply equation~\eqref{eq:uav}, expand the numerator and apply
Lemma~\ref{lemma:dis}(a).   
We obtain a linear combination of expressions like
  \beq \frac 1s\sum_{\zeta^s=1}\left.\frac{(\zeta x)^e}{(1-\zeta
    x)(1-\zeta^2 x^2)\cdots (1-\zeta^k x^k)}\right|_{x\to
    x^{1/s}} = \left.G(x)\right|_{x\to x^{1/s}}, \label{eq:sumzeta}
  \eeq 
say.  Let $\zeta_s=e^{2\pi i/s}$, a primitive $s$th root of unity. The
order to which 1 is a pole in equation~\eqref{eq:sumzeta} is thus at
most the order to which $\zeta_s$ is a pole of $G(x)$.
Now any term indexed by
$\zeta\neq 1$ has $\zeta_s$ as a pole of $G(x)$ of order less than $k$,
while the term indexed by $\zeta=1$ has a pole of order at most $k$ at
$x=1$. Hence if in the end we have a pole of order $k$, then it 
suffices to retain only the term in \eqref{eq:sumzeta} indexed by
$\zeta=1$. Therefore if, for any integer $e$,
 $$  \frac 1s\sum_{\zeta^s=1}\left.\frac{(\zeta x)^e}{(1-\zeta
    x)(1-\zeta^2 x^2)\cdots (1-\zeta^k x^k)}\right|_{x\to
    x^{1/s}} = \frac{c_0}{(1-x)^k}+
    O\left(\frac{1}{(1-x)^{k-1}}\right) $$ 
for $x\to 1$, then 
  $$ c_0  =  \lim_{x\to 1}(1-x^s)^k \frac{x^e}{s(1-x)(1-x^2)\cdots
       (1-x^k)} =  \frac{s^{k-1}}{k!}. $$
\indent Write
  $$ (1-q^{vb+r+k})(1-q^{vb+r+k-1})\cdots (1-q^{vb+r+1}) =
    \frac{\sum_{i=0}^k(-1)^i Q_i(q)q^{bvi}}{[k]!}, $$
where $Q_i(q)$ is a polynomial independent of $b$ and $v$, so 
$Q_i(1)=\binom ki$. Note that $u-vi\geq 0$ if and only if
$i\leq\lfloor \alpha\rfloor$. It follows that 
  \beas G_{\alpha,k,r}(x) & = & \sum_{b\geq 0}\left[ q^{ub+d-c}\right]
  \frac{\sum_{i=0}^k (-1)^i 
    Q_i(q)q^{bvi}}{[k]!}x^{bv+r}\\ & = &
   \sum_{b\geq 0}[q^{(u-vi)b}]\frac{\sum_{i=0}^k (-1)^i 
    Q_i(q)q^{c-d}}{[k]!}x^{bv+r}\\
   & = & \left(\frac{1}{k!}
    \sum_{i=0}^{\lfloor\alpha\rfloor}(-1)^i\binom ki(u-vi)^{k-1} \right)
    \frac{x^r}{(1-x^v)^k}+O\left(\frac{1}{(1-x)^{k-1}}\right)\\ & = & 
    \left(\frac{1}{k!}
    \sum_{i=0}^{\lfloor\alpha\rfloor}(-1)^i\binom ki(u-vi)^{k-1} \right)
    \frac{1}{v^k(1-x)^k}+O\left(\frac{1}{(1-x)^{k-1}}\right).
  \eeas
  
Now sum over $0\leq r<v$. Since we have $v$ terms in the sum, we pick up
an extra factor of $v$ on the right, giving
  \beas  \sum_{a\geq 0} [q^{\lfloor ua/v\rfloor}]\binom{a+k}{k}_q x^a 
   & = &  \left(\frac{1}{k!}
    \sum_{i=0}^{\lfloor\alpha\rfloor}(-1)^i\binom ki(u-vi)^{k-1} \right)
    \frac{1}{v^{k-1}(1-x)^k}+O\left(\frac{1}{(1-x)^{k-1}}\right)\\
    & = & \left(\frac{1}{k!} \sum_{i=0}^{\lfloor\alpha\rfloor}
    (-1)^i\binom ki(\alpha-i)^{k-1} \right) 
    \frac{1}{(1-x)^k}+O\left(\frac{1}{(1-x)^{k-1}}\right). \eeas
    
Now
  \beas [x^a]\frac{1}{(1-x)^k} & = & \binom{k+a-1}{k-1} =
  \frac{a^{k-1}}{(k-1)!} + O(a^{k-2}), \eeas 
completing the proof for $\alpha$ rational.

The proof for general $\alpha$ now follows by a simple continuity
argument, using the unimodality and symmetry of the coefficients of
$\binom{a+k}{k}_q$.   
\end{proof}

The numbers $C(\alpha,k)$ have appeared before and are known as 
\emph{Euler-Frobenius numbers}, 
denoted $A_{k-1,\lfloor \alpha\rfloor,\alpha-\lfloor
  \alpha\rfloor}$. For a discussion of the history and properties of
  these numbers, see Janson \cite{janson}. 
Some special cases are of interest. 
Recall that the \emph{Eulerian number} $A(d,i)$ can
be defined as the number of permutations $w$ of $1,2,\dots,d$ with
$i-1$ descents (e.g.\ \cite[{\S}1.4]{St0}). Similarly the
\emph{MacMahon number} $B(d,i)$ can be defined as the number of
elements in the hyperoctahedral group $B_n$ according to the number of
type $B$ descents. For further information, see \cite{oeis}. Standard
results about these numbers imply that for integers $1\leq j<k$,
  \beas C(j,k) & = & A(k-1,j),\\
        2^{k-1}C((2j-1)/2,k) & = & B(k-1,j). \eeas

There is an alternative way to show the above formula for
$C(\alpha,k)$ (done with assistance from Fu Liu). 
Write $\beta=\lfloor \alpha\rfloor$. Since the coefficient of $q^{a\beta}$
in $\binom{a+k}{k}_q$ is the number of partitions of $a\beta$ into at most $a$ parts of length at most $k$, equivalently, it is equal to the number of solutions
$(m_1,\dots,m_k)$ in nonnegative integers to
   \beas m_1+2 m_2+\cdots +k m_k & = & a\beta,\\
       m_1+\cdots+m_k & \leq & a. \eeas 
       
Set $x_i=m_i/a$ and let $a\to\infty$. Standard arguments (see e.g.,
\cite[Proposition 4.6.13]{St0}) show that $C(\alpha,k)$ is the
$(k-1)$-dimensional relative volume (as defined in \cite[p.~497]{St0})
of the convex polytope: 
  \beas  x_i & \geq & 0,\ \ 1\leq i\leq k,\\
    x_1 + 2x_2 +\cdots +kx_k & = & \beta,\\
    x_1+x_2+\cdots+x_k & \leq & 1. \eeas
    
Set $y_i=x_i+x_{i+1}+\cdots+x_k$. The matrix of this linear
transformation has determinant 1, so it preserves the relative
volume. We get the new polytope $\mathcal{P}_k$ defined 
by
    $$ y_1 + y_2 + \cdots + y_k = \beta,$$
\vspace{-1.5em}
     $$ 0\leq y_1\leq y_2\leq\cdots\leq y_k\leq 1. $$
     
By symmetry, the relative volume of $\mathcal{P}_k$
is $1/k!$ times the relative volume of the polytope
    $$ y_1 + y_2 +\cdots +y_k = \beta, $$
\vspace{-1.5em}
    $$ 0\leq y_i\leq 1,\ \ 1\leq i\leq k. $$
    
This polytope is a cube cross-section, whose relative volume is
computed e.g.\ in \cite[Theorem 2.1]{janson}, completing the proof.

When $\alpha\in\Q$,  $C(\alpha,k)$ is related to the
Eulerian polynomial $A_{k-1}(x)$ via the following result.

\begin{proposition} \label{prop:geneul}
Let $v\in\PP$. Then
  \beq v^{k-1}\sum_{u\geq 0}C(u/v,k)x^u = (1+x+x^2+\cdots+x^{v-1})^k
      A_{k-1}(x). \label{eq:cucgf} \eeq
\end{proposition}

\begin{proof}
We have
  \beq v^{k-1}\sum_{u\geq 0}C(u/v,k)x^u = \sum_{u\geq 0}
      \sum_{i=0}^{\lfloor u/v\rfloor}(-1)^i\binom ku
      (u-vi)^{k-1}x^i. \label{eq:cgf} \eeq
      
A fundamental property of Eulerian polynomials is the identity
(see \cite[Proposition 1.4.4]{St0})
  $$ \sum_{n\geq 0} n^{k-1}x^n =\frac{A_{k-1}(x)}{(1-x)^k}. $$
  
Hence,
  \beq A_{k-1}(x)(1+x+\cdots+x^{v-1})^k = (1-x^v)^k 
     \sum_{n\geq 0} n^{k-1}x^n. \label{eq:akv} \eeq
     
It is now routine to compute the coefficient of $x^m$ on the right-hand
sides of equations~\eqref{eq:cgf} and \eqref{eq:akv} and see that they
agree term by term.
\end{proof}

Note that if $j\in\PP$ and we take the coefficient of $x^{jv}$ on both
sides of equation~\eqref{eq:cucgf}, then we obtain the identity
  $$ v^{k-1}A(k-1,j)=[x^{vj}](1+x+x^2+\cdots+x^{v-1})^kA_{k-1}(x). $$
It is not difficult to give a direct proof of this identity.


Let us now turn to the \emph{difference} between two consecutive
coefficients of $\binom{a+k}{k}_q$, i.e., the function $f_{k,c}(a)$ of
equation~\eqref{eq:fkc}. We consider here only the
coefficients near the middle (i.e., $q^{aj}$) when $k=2j$, though
undoubtedly our results can be extended to other coefficients. Note
that, by the previous theorem, we have 
  $$ [q^{aj-c}]\binom{a+k}{k}_q\sim
  [q^{aj-c-1}]\binom{a+k}{k}_q\sim \frac{1}{(k-1)!\,k!}C(\alpha,
  k)a^{k-1},\ \ a\to\infty. $$ 
  
Thus we might expect that the difference
$([q^{aj-c}]-[q^{aj-c-1}])\binom{a+k}{k}_q$ grows like
$a^{k-2}$. However, the next result shows that the correct
growth rate is $a^{k-3}$. 

\begin{theorem} \label{thm:dk}
Let $c\in\N$ and $k=2j$, where
$j\in\mathbb{P}$. Then for $j\geq 3$ we have
 $$ \left([q^{aj-c}]-[q^{aj-c-1}]\right)\binom{a+k}{k}_q =
     \frac{2c+1}{(k-3)!\,k!}D(k)a^{k-3}+O(a^{k-4}), 
 $$
where 
  $$ D(k)= \frac 12\sum_{i=0}^{j-1}(-1)^{i+1}\binom ki 
    (j-i)^{k-3}. $$ 
\end{theorem}

\begin{proof}
Write 
  \beas F_k(x,t) & = & \sum_{a\geq 0}\sum_{c\geq
    0}([q^{aj-c}]-[q^{aj-c-1}]) \binom{a+k}{k}_q x^at^c\\ & = & 
    \sum_{a\geq 0}[q^{aj}]\frac{(1-q^{a+k})\cdots (1-q^{a+1})}
        {(1-q^k)\cdots (1-q^2)(1-qt)}x^a. \eeas
        
When $k\geq 6$, the order to which a primitive root of unity
$x\neq 1$ is a pole of $F_k(x,t)$ is at most $k-3$. Thus we
need to show that the pole at $x=1$ contributes the stated
result. 

Let 
  $$ F_k(x,t) = \alpha_k(t)\frac{1}{(1-x)^{k-1}} + \beta_k(t)
     \frac{1}{(1-x)^{k-2}} +O\left(\frac{1}{(1-x)^{k-3}}\right). $$
First we show that $\alpha_k(t)=0$. Reasoning as in the proof of
Theorem~\ref{thm:calk} gives 
  $$ \alpha_k(t) = \frac{1}{(k-2)!\,k!(1-t)}
      \sum_{i=0}^{j-1} (-1)^i\binom ki (j-i)^{k-2}. $$
      
Since $k$ is even, the summand $(-1)^i\binom ki(j-i)^{k-2}$ remains the
same when we substitute $k-i$ for $i$. Moreover, when $i=j$ the
summand is 0. Hence
  $$ \alpha_k(t) = \frac{1}{2(k-2)!\,k!(1-t)}
      \sum_{i=0}^k (-1)^i\binom ki (j-i)^{k-2}. $$
      
This sum is the $k$th difference at 0 of a polynomial of degree $k-2$,
and is therefore equal to 0 (see \cite[Proposition 1.9.2]{St0}), as desired.

We now need to find the coefficient $\beta$ of $(1-x)^{k-2}$ in the 
Laurent expansion at $x=1$ of linear combinations of rational
functions of the type
  $$ H=\frac{P(x)}{(1-x^2)\cdots (1-x^k)(1-xt)} = 
   \frac{\alpha}{(1-x)^{k-1}} + 
   \frac{\beta}{(1-x)^{k-2}} + \cdots, $$
where $P(x)$ is a polynomial in $x$. Write
$(i)_x=1+x+x^2+\cdots+x^{i-1}$. It is easy to see that
$\alpha=P(1)/k!(1-t)$. Thus  
  \beas \beta & = & \lim_{x\to 1}(1-x)^{k-2}
     \left( \frac{P(x)}{(1-x^2)\cdots (1-x^k)(1-xt)} -
   \frac{P(1)}{k!(1-t)(1-x)^{k-1}}\right)\\ 
   & = &
    \lim_{x\to 1}\frac{1}{1-x}\cdot \frac{P(x)k!(1-t) -
    P(1)(2)_x\cdots(k)_x(1-xt)}{(2)_x\cdots(k)_x(1-xt)k!(1-t)}\\ 
   & = &
   -\frac{1}{k!^2(1-t)^2}\frac{d}{dx}\left.\left(P(x)k!(1-t)-
     P(1)(2)_x\cdots(k)_x(1-xt)\right)\right|_{x=1}\\ 
   & = &
     -\frac{1}{k!^2(1-t)^2}\left( P'(1)k!(1-t)-k!P(1)\left(\frac 12
     +\frac 33+\cdots+\frac{\binom k2}{k}\right)(1-t)+P(1)t\right)\\
   & = & -\frac{1}{k!(1-t)^2}\left( P'(1)(1-t)-\frac 12P(1)\binom
     k2(1-t) +P(1)t\right)\\ 
   & = & 
    \frac{1}{k!(1-t)^2}\left( -P'(1)(1-t)
     +\frac 14P(1)(-k+kt+k^2-k^2t)-P(1)t\right). \eeas
     
Let us apply this result to $P(x)=P_i(x)$, where $P_i$ is defined by
equation~\eqref{eq:piq2}. Clearly $P_i(1)=\binom ki$, while 
   $$ P'_i(1) = \sum_{\substack{S\subseteq [k]\\ \#S=i}}
     \sum_{s\in S} s. $$
     
The element $i\in [k]$ appears in $\binom{k-1}{i-1}$ $i$-element
subsets of $[k]$. Hence
  $$ P'_i(1) = \sum_{i=1}^k i\binom{k-1}{i-1} = 
    \binom{k+1}{2} \binom{k-1}{i-1}, $$
where when $i=0$ we set $\binom{k-1}{-1}=0$. Arguing as in the proof
of Theorem~\ref{thm:calk} now gives 
 $$ \!\!\!\!\!\!\!\!\!\!\!\!\!\!\!\!\!\!\! 
   \beta_k(t) = \frac{1}{k!(1-t)^2}\sum_{i=0}^{j-1} 
    (-1)^{i+1}(j-i)^{k-3} $$
\vspace{-1.2em}
 $$ \times 
    \left(\binom ki\left(\frac 12(1-t)(k-1)(j-i-1)+t-
      \frac 14(1-t)k(k-1)\right)\right. $$ 
\vspace{-1.2em}
   \beq +\left.(1-t)\binom{k+1}{2}
       \binom{k-1}{i-1}\right). \label{eq:fkxt} \eeq
       
If we set $t=-1$ on the right-hand-side of equation~\eqref{eq:fkxt}, then a straightforward computation shows that the sum is 0. If 
we set $t=1$, then another computation gives 
   $$ \sum_{i=0}^{j-1}(j-i)^{k-3}(-1)^{i+1}\binom ki. $$ 
   
Since
  $$ \frac{1+t}{(1-t)^2} = \sum_{c\geq 0}(2c+1)t^c, $$ 
the proof now follows. 
\end{proof}

\begin{remark}
\begin{enumerate}
\item[(a)] It follows from work of Verma \cite{verma} and of
Hering and Howard \cite{h-h} that $D(k)$ also satisfies
  \beq K_{a(k/2,k/2),a\cdot 1^k} = \frac{1}{(k-3)!}D(k)a^{k-3} +
   O(a^{k-4}), \label{eq:kostka} \eeq
where $K_{\lambda\mu}$ is a Kostka number and $a\cdot 1^k$ denotes the
partition of $ak$ with $k$ $a$'s. Is the appearance of $D(k)$ in both
Theorem~\ref{thm:dk} and equation~\eqref{eq:kostka} just a
  coincidence? 

  

\item[(b)] Theorem~\ref{thm:dk} is false for $j=2$. Indeed, it follows from
equation~\eqref{eq:f4xt} that
  $$ F_4(x,t) = \frac{1-t+t^2}{6(1-t)(1-t^2)}\cdot \frac{1}{(1-x)^2} 
     +O\left(\frac{1}{1-x}\right) $$
and 
  $$ \left([q^{2a-c}]-[q^{2a-c-1}]\right)\binom{a+4}{4}_q =
   \frac{1}{24}(2c+1+3\cdot(-1)^c)a +O(1),\ \ a\to\infty. $$

\item[(c)] An obvious problem arising from our work is the extension
  of Theorem~\ref{thm:calk} to additional terms. Can such a
  computation be automated?
\end{enumerate}
\end{remark}

\section{Acknowledgements} We are grateful to several anonymous
reviewers for several comments, to Fu Liu for her assistance with the
proof presented after Theorem \ref{thm:calk}, and to Qinghu Hou for
pointing out some errors in our original computations and for noting
that the Omega Package \cite{omega} of Andrews, Paule, and Riese can
be used very effectively for the computation of $F_k(x,t)$ and
$G_k(x,t)$. The second author warmly thanks the first author for his
hospitality during calendar year 2013 and the MIT Math Department for
partial financial support.


\end{document}